\documentclass[12pt]{amsart}
\usepackage[utf8]{inputenc}
\usepackage[T1]{fontenc}
\usepackage{amsthm}
\usepackage{amssymb}
\usepackage{amsfonts}
\usepackage{interval}
\intervalconfig{soft open fences}
\usepackage{cite}
\usepackage{enumerate}
\usepackage{enumitem}
\usepackage{hyperref}
\usepackage[noabbrev,capitalise]{cleveref}
\usepackage{geometry}
\usepackage{fullpage}
\usepackage{bbm}
\usepackage{color}
\usepackage{dirtytalk}
\usepackage[draft]{todonotes}
\usepackage[right]{lineno}
\usepackage{setspace}

\newtheoremstyle{case}{}{}{\normalfont}{}{\itshape}{:}{ }{}

\newcommand{\oldqed}{}
\def\endofFact{\hfill\scalebox{.6}{$\Box$}}
\newenvironment{claimproof}[1][Proof]{
	\renewcommand{\oldqed}{\qedsymbol}
	\renewcommand{\qedsymbol}{\endofFact}
	\begin{proof}[#1]
	}{
	\end{proof}
	\renewcommand{\qedsymbol}{\oldqed}
}

\title{on the anti-Ramsey threshold for non-balanced graphs}

\author[P.~Campos]{Pedro Ara\'{u}jo}
\author[T.~Martins]{Ta\'{\i}sa Martins}
\author[L.~Mattos]{Let\'{\i}cia Mattos}
\author[W.~Mendon\c{c}a]{Walner Mendon\c{c}a}
\author[L.~Moreira]{Luiz Moreira}
\author[G.~O.~Mota]{Guilherme O. Mota}

\address{Institute of Computer Science of the Czech Academy of Sciences, Pod Vodárenskou věží 2, 18207, Prague, Czech Republic (P.~Araújo)}
\email{araujo@cs.cas.cz}

\address{Instituto de Matem\'atica, Universidade Federal Fluminense, 
	Niter\'oi, Brazil (T.~Martins)}
\email{tlmartins@id.uff.br}

\address{Freie Universität Berlin and Berlin Mathematical School (BMS/MATH+), Arnimallee 3, 14195 Berlin, Germany (L.~Mattos)}
\email{lmattos@zedat.fu-berlin.de}

\address{IST Austria, Am Campus 1, 3400 Klosterneuburg, Austria (W.~Mendon\c{c}a)}
\email{wmendonc@ist.ac.at}

\address{Guaca Macramê, Rua do Z 33, 20251-600, Rio de Janeiro, RJ, Brazil (L.~Moreira)}          
\email{lzplfm@gmail.com}

\address{Instituto de Matem\'atica e Estat\'{\i}stica, Universidade de
	S\~ao Paulo, Rua do Mat\~ao 1010, 05508-090 S\~ao Paulo, Brazil (G.~Mota)}
\email{mota@ime.usp.br}

\thanks{%
	P. Araújo was supported by the long-term strategic development financing of the Institute of Computer Science (RVO: 67985807) and by the Czech Science Foundation, grant number GJ20-27757Y.
	L.~Mattos was supported by CAPES and by the Deutsche Forschungsgemeinschaft (DFG, German Research Foundation) under Germany's Excellence Strategy – The Berlin Mathematics Research Center MATH+ (EXC-2046/1, project ID: 390685689). 
	W.~Mendon\c{c}a was supported by CAPES (88882.332408/2010--01).
    L. Moreira was supported by CNPq.
	G. O. Mota was supported by CNPq (306620/2020-0,
	428385/2018-4) and FAPESP (2018/04876-1, 2019/13364-7).
}

\newtheorem{thm}{Theorem}[section]

\newtheorem{lemma}[thm]{Lemma}

\newtheorem{cor}[thm]{Corollary}

\newtheorem{claim}[thm]{Claim}
\newtheorem{prop}[thm]{Proposition}

\newcommand{\Gs}{\mathcal{G}^{*}}
\newcommand{\flecha}{\overset{\mathrm{rb}}{\longrightarrow}}
\newcommand{\naoflecha}{\overset{\mathrm{rb}}{\nrightarrow}}

\newcommand{\prb}[1]{\mathrm{p}^{\mathrm{rb}}_{#1}}

\newcommand{\gnp}{G\left(n,p\right)}

\newcommand{\se}{\subseteq}
\newcommand{\disc}{\mathrm{DISC}}

\def\le{\leqslant}
\def\leq{\leqslant}
\def\ge{\geqslant}
\def\geq{\geqslant}
\def\A{\mathcal{A}}

\def\G{\mathcal{G}}
\def\E{\mathbb{E}}
\def\cE{\mathcal{E}}

\def\P{\mathbb{P}}
\def\cP{\mathcal{P}}

\def\eps{\varepsilon}

\begin{document}
	\begin{abstract}
		For graphs $G,H$, we write $G \flecha H $ if any proper 
		edge-coloring of $G$ contains a rainbow copy of $H$, i.e., a copy 
		where no color appears more than once.
		Kohayakawa, Konstadinidis and the last author proved that the threshold for $G(n,p) \flecha H$ is at most $n^{-1/m_2(H)}$.
		Previous results have matched the lower bound for this anti-Ramsey threshold for cycles and complete graphs with at least 5 vertices.
		Kohayakawa, Konstadinidis and the last author also presented an infinite family of graphs $H$ for which the anti-Ramsey threshold is asymptotically smaller than $n^{-1/m_2(H)}$.
		In this paper, we devise a framework that provides a richer and more complex family of such graphs that includes all the previously known examples.
	\end{abstract}
	
	\maketitle
	
	\section{Introduction}
	
	We say that a graph $G$ has the anti-Ramsey property $G\flecha H$ if every proper edge-coloring of $G$ contains a rainbow copy of $H$. The study of anti-Ramsey properties can be traced back to a question of Spencer, mentioned by Erd\H os in~\cite{Erdos79}: Does there exist a graph with arbitrarily large girth and such that every proper edge-coloring contains a rainbow cycle? Rödl and Tuza answered this question affirmatively \cite{RT} by showing that $G(n,n^{\varepsilon-1})\flecha C_k$ for large $k$ and by deleting one edge from each small cycle.

	As $G\flecha H$ is an increasing property, it admits a threshold function $\prb{H}=\prb{H}(n)$, which is the focus of this work. Kohayakawa, Konstadinidis and the last author~\cite{KKM14} proved that, for any fixed graph $H$, we have $\prb{H} \leq n^{-1/m_2(H)}$, where

    \begin{equation}
		\label{eq:m2density}
		m_{2}(H):=\max\left\{ \frac{e(J)-1}{v(J)-2} : J \subseteq H, v(J) \geq 3 \right \}
	\end{equation}
	is the $m_{2}$-density of the graph $H$.
	Nenadov, Person, \v{S}kori{\'c} and Steger~\cite{NPSS} showed that this upper bound is sharp for cycles with at least $7$ vertices and complete graphs with at least $19$ vertices.
	This result was extended for cycles and cliques with at least $5$ vertices, in~\cite{BCMP} and~\cite{KMPS}, respectively.
	
	Apart from cliques and cycles, not much is known about $\prb{H}$.
	One might feel compelled to conjecture that indeed this threshold is determined by the $m_2$-density
	for all graphs, specially because of `standard' Ramsey threshold results such as the one from Rödl and Ruciński~\cite{RR95}. However, the
	authors of~\cite{KKM18} proved that this is not the case for a fairly large class of graphs, which
	is an evidence of the inherent difference between anti-Ramsey and standard Ramsey properties. To illustrate this difference, we consider the case of $H=K_3$. Note that every proper-coloring of a triangle is rainbow and therefore the threshold for the event $G(n,p)\flecha K_3$ is the same as the threshold for the appearance of triangles, which is a local property. For some other graphs $H$, such as cliques and cycles with more than $3$ vertices, it turns out that the property $G(n,p)\flecha H$ seems to be related with more global aspects of the host graph. In this paper we explore the interplay of these two cases.
	
	The class of graphs in~\cite{KKM18} consists of graphs obtained by `attaching' a triangle to an edge of a graph $H$ with $m_2(H)<2$, a result that we extend by allowing different graphs to be attached to $H$.
	Before we state our results, we formally define an \emph{amalgamation} of graphs, which we denote by $\oplus$.
	A \emph{$2$-labeled graph} is a graph in which exactly 
	one vertex is labeled $1$, exactly one vertex is labeled $2$ and they form an edge.
	For any $2$-labeled graphs $F$ and $H$, we define $F\oplus H$ as the graph obtained by taking the disjoint union of $F$ and $H$ and identifying the
	vertices with label $1$ and the vertices with label $2$ together with the respective edge.
	An illustration of this definition is depicted on~\cref{fig:amalg}.

	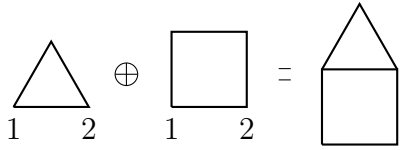
\begin{figure}
		\begin{tikzpicture}
			\draw[thick,black] (0,0) node [below] {$1$} -- (1,0) node [below] {$2$} -- (0.5,0.87) -- (0,0);
			\draw (1.5,0.435) circle [radius=0.15];
			\draw (1.5,0.285) -- (1.5,0.585);
			\draw (1.35,0.435) -- (1.65,0.435);
			\draw[thick,black] (2.1,0)  node [below] {$1$} -- (3.1,0) node [below] {$2$} -- (3.1,1)  -- (2.1,1) -- (2.1,0);
			\draw (3.525,0.51) -- (3.685,0.51);
			\draw (3.525,0.36) -- (3.685,0.36);
			\draw[thick,black] (4.1,-0.5) -- (5.1,-0.5) -- (5.1,0.5)  -- (4.1,0.5) -- (4.1,-0.5);
			\draw[thick,black] (5.1,0.5) -- (4.6,1.37) -- (4.1,0.5);
		\end{tikzpicture}
		\caption{An example of amalgamation.}\label{fig:amalg}
	\end{figure}

	We say that a graph $S$ is $2$-balanced if the maximum in \eqref{eq:m2density} is attained by $S$ itself. We are now ready to state our main theorem.
	
	\begin{thm}\label{thm:main}
		Let $H$, $F$ be $2$-labeled graphs with $1 < m_2(H) < m_2(F)$. For any $2$-balanced graph $S$ such that $S \flecha F$, there exists $C>0$ such that

        \begin{align*}
			\lim_{n \rightarrow \infty} \P\left[G(n,p) \flecha F \oplus H\right] = 1, \text{ if }p\geq Cn^{-\beta(H,S)},
		\end{align*}
		where
		\[
          \beta(H,S)=\dfrac{1}{e(S)}\left(v(S)-2+\dfrac{1}{m_2(H)} \right).
        \]
	\end{thm}
	
	We remark that the definition of $\beta$ comes from a more general definition from~\cite{KK}. This parameter is defined to ensure that the number of copies of $S$ in $G(n,p)$ is of order $\Omega(n^{2-1/m_2(H)})$, which is essential for our proof (see Section 2).
	
	As an application of Theorem~\ref{thm:main} we present an infinite family of graphs $H$ such that $n^{-1/m_2(H)}$ is not a threshold for $G(n,p)\flecha H$.
	For any $t \in \mathbb{N}$, let $B_t$ the $t$-book graph, consisted of $t$ triangles sharing a common edge. We prove that $B_{3t-2}\flecha B_t$ and, since $m_2(B_t)=2$, that $B_t\oplus H$ satisfies the hypothesis of Theorem~\ref{thm:main} for every positive integer $t$ and every graph $H$ with $m_2(H) \in (1,2)$.
	
	\begin{cor}
		\label{booksaregreat}
		For every $t\in \mathbb{N}$ and every graph $H$ with $m_2(H)\in (1,2)$, we have
		\[
          p_{B_t\oplus H}^{\mathrm{rb}}=o\big(n^{-1/m_2(B_t\oplus H)}\big)
        \]
		
		\noindent for any $2$-labeled copies of $B_t$ and $H$.
	\end{cor}

	The paper is organized as follows:
	in Section~\ref{sec:overview} we give an overview of the proof of Theorem~\ref{thm:main} ;
	in Section~\ref{sec:tools} we recall some definitions and results on the Regularity Method;
	In Section~\ref{sec:colorpartition} we explore properties of proper colorings of $G(n,p)$;
	in Section~\ref{section:main} we prove Theorem~\ref{thm:main}; and
	in Section~\ref{sec:split} we prove Corollary \ref{booksaregreat}.

	\section{Overview of the proof}\label{sec:overview}
	
	We aim to show that, under the hypothesis of Theorem~\ref{thm:main}, there exists $C > 0$ such that, for $p\geq Cn^{-\beta(H,S)}$, every proper edge-coloring of $G(n,p)$ contains a rainbow copy of $F\oplus H$ with high probability.
	A natural strategy is to partition the colors into $e(F \oplus H)$ classes and find a copy of $F\oplus H$ so that each of its edges belongs to a different class and, consequently, is rainbow. One good attempt would be to make a random partition of the colors, while trying to prove that the graphs induced by each class of colors is pseudo-random.
	This is indeed the approach used in~\cite{KKM14} to provide an upper bound on the
	threshold of this anti-Ramsey property for any graph.
	
	However, such a general strategy cannot work since we are dealing with lower densities than the one determined by $m_2(F\oplus H)$, which is a barrier to applying an embedding lemma for random graphs such as Theorem \ref{thm:KLR}, formerly known as the K\L R Conjecture.
	Moreover, by the results of~\cite{NPSS,BCMP, KMPS} we know that the $m_2$-density in fact determines the threshold for cliques and cycles, so if we work with lower densities we must show that the
	imbalance of $F\oplus H$ plays a role in our proof. Recall that we consider a graph $S$ such that $S\flecha F$. By finding many copies of $S$, we will get many rainbow copies of $F$, which we aim to extend to a rainbow copy of $F\oplus H$ using the Theorem \ref{thm:KLR}.

	Given a realization $G$ of $G(n,p)$ and a proper edge-coloring $c:E(G)\rightarrow \mathbb{N}$, for each color $i \in e(E(G))$ we assign
	independently and uniformly at random an edge $\sigma(i) \in E(H)$.
	Let $V(H)=\{u_1,u_2,\ldots, u_h\}$ be a fixed labeling of
	the vertices of $H$.
	For each $u_iu_j \in E(H)$, we denote by $G_{u_iu_j}$ the graph with the edge set $E(G_{u_iu_j}) = \{e \in E(G): \sigma(c(e)) = u_iu_j\}$.
	That is, $E(G_{u_i u_j})$ are those edges of $E(G)$ for which their color was assigned to $u_i u_j$.

	Let $V(S)=\{v_1,v_2\cdots,v_s\}$ be a fixed labeling of
	the vertices of $S$ in which $\{v_1,\cdots,v_f\}$ form a copy of $F$.
	Consider $F \oplus H$ and $S \oplus H$ to be the amalgamations obtained by identifying $u_{1}$ with $v_{1}$ and $u_{2}$ with $v_{2}$.
	We will look for a copy $S' \oplus H'$ of
	$S \oplus H$ in $G$ such that the following holds.
	If $e \in E(H')$ is identified with $u_iu_j$, then $c(e) \in \sigma^{-1}(u_iu_j)$;
	if $e \in E(S')$, then $c(e) \in \sigma^{-1}(u_1u_2)$.
	As $S \flecha F$ and the sets $\sigma^{-1}(u_iu_j)$ are disjoint, such colored copy of $S \oplus H$ will give us a rainbow
	copy of $F \oplus H$.

	We start by fixing an equitable partition of the vertices of $G$ as follows:

    \begin{align*}
		V(G) =\bigcup_{i=1}^{s} V_i \cup \bigcup_{i=3}^{h} U_i.
	\end{align*}
	For each $i$, the set $V_i$
	corresponds to the vertex $v_{i}$ in $S$
	and the set $U_i$ corresponds to the vertex $u_{i}$
	in $H$.
	By convenience, we also set $U_{1} = V_{1}$ and $U_{2} = V_{2}$.
	Denote by $G[V_{1},\ldots,V_{s}]$ the graph formed by
	the edges of $G$ whose endpoints belong to different
	sets $V_{i}$ and $V_{j}$.
	First, we prove that with high probability there exist

    \begin{align}\label{eq:betareason}
		\Omega( n^{v(S)} p^{e(S)} ) = \Omega(n^{2 - 1/m_2(H)} )
	\end{align}
	edge-disjoint transversal copies of $S$ in $G[V_1,\ldots,V_s]$, that is, copies of $S$ in which each vertex
	belongs to a different set $V_{i}$.
	Since $S$ is $2$-balanced and $m_{2}(S)>m_{2}(H)$, we have $p\leq n^{-1/m_2(S)}$ and hence most copies of $S$ in $G$ are \emph{isolated}, that is,
	most copies do not share edges with any other copies.
	We guarantee that a positive proportion of these copies are present in $G_{u_1u_2}$ by applying Azuma's inequality (cf. Lemma~\ref{lemma:azuma}).
	In this argument we control the influence of each color in the number of copies of $S$ in $G_{u_1u_2}$ by recalling that each color forms a matching and using that the copies are isolated.

	Let $G_{u_{1}u_{2}}^{S}[V_{1},\ldots,V_{s}]$ be the graph induced by the edges of
	the isolated transversal copies of $S$ in $G_{u_{1}u_{2}}[V_{1},\ldots,V_{s}]$.
	We show that there exists a bipartite subgraph
	$B_{12} = B_{12}[W_{1},W_{2}]$ of $G_{u_{1}u_{2}}^{S}[V_{1},V_{2}]$
	which is $(\varepsilon,q)$-regular and has $\lfloor\alpha n^{2}q\rfloor$ edges, where $\alpha,\varepsilon \in (0,1)$ and $q = B^{e(S)}n^{-1/m_{2}(H)}$.
	The density of $B_{12}$ comes from~\eqref{eq:betareason}, 
	while the regularity comes from Lemma~\ref{tools:lemma2} combined with Lemma~\ref{lemma:upperuniff}.
	Relabeling the sets $V_{1},\ldots,V_{s}$ if
	necessary, we also show that each $e \in E(B_{12})$ is contained in an isolated copy of $S$ such that $e$ corresponds to $u_{1}u_{2}$.
	Moreover, the sets $W_{1} \subseteq V_{1}$ and $W_{2} \subseteq V_{2}$ have equal size and contain a constant
	proportion of the vertices in $V_{1}$ and $V_{2}$,
	respectively.
	%For more details, see Claim~\ref{claim:regF}.
	
	Now, the plan is to apply Theorem \ref{thm:KLR}
	to find a rainbow copy of $H$ in $G[U_1,\ldots,U_h]$
	which contains one edge in $B_{12}$
	and, apart from this edge, only uses colors outside $c^{-1}(u_1u_2)$.
	To do so, we first prove that for all $j > 2$ and $i < j$ the following holds.
	There exists a bipartite subgraph
	$B_{ij} = B_{ij}[W_{i},W_{j}]$ of $G_{u_{i}u_{j}}(U_{i},U_{j})$
	which is $(\varepsilon,q)$-regular and has $\lfloor\alpha n^{2}q\rfloor$ edges with $|W_{i}| = |W_{1}|$ for all $1 \le i \le h$.
	The density of $B_{ij}$ comes from the
	concentration of the degrees in $G_{u_{i}u_{j}}$ 
	(cf.~Lemma~\ref{lemma:concdeg}), using that $m_2(H)>1$ and consequently that $q\gg\log n/n$.
	To prove $(\varepsilon,q)$-regularity, we use Azuma's inequality to
	control the number of cycles of a fixed even length 
	and apply a well-know theorem of
	Chung and Graham (cf. Lemma~\ref{lemma:ChGr}).
	For more details, see Lemmas~\ref{lemma:disc} and~\ref{lemma:discGnp}.
	
	By choosing $B$ larger enough, we apply Theorem \ref{thm:KLR} to deduce that the probability that $\cup_{ij} B_{ij}$ does not contain a transversal copy of $H$ is at most of order $\exp(-qn^2)$. Since $m_{2}(H)> 1$ and consequently $qn^{2} \gg n$, this bound is enough to account for the number of choices for the sets $W_{1},\ldots,W_{h}$.
	This implies that the desired rainbow copy of $H$ in 
	$G[U_{1},\ldots,U_{h}]$ exists with high probability.
	As $S \flecha F$ and each $e \in E(B_{12})$ is contained in a copy of $S$ so that $e$ corresponds to $u_{1}u_{2}$, 
	we conclude that there is a rainbow copy of $F \oplus H$ in $G$.

	\section{Tool box}\label{sec:tools}
	%quem mudar o nome dessa seção é um feioso.
	
	For any graph $G$ and disjoint subsets $U,V\se V(G)$, 
	define the \emph{density} of the pair $(U,V)$ in $G$ to be
	\[
      d_{G}(U,V) = \frac{e_{G}(U,V)}{|U||V|},
	\]
	where $e_{G}(U,V)$ denotes the number of edges across $U$ and $V$.
	We suppress $G$ from the notation whenever it is clear from context.

	For any $\mu,p \in \mathbb{R}$, we say that $G$ is $(\mu,p)$-\emph{upper uniform} if
	\[
      d_{G}(U,V)\leq (1+\mu)p,
	\]
	for every disjoint pair of sets $U,V\subseteq V(G)$ such that $|U|,|V|\geq \mu v(G)$.
	If $G=G[U,V]$,
	we say that $G$ is $(\eps,p)$-\emph{regular} if
	\[
      |d_{G}(U,V) - d_{G}(U',V')| \leq \eps p,
	\]
	for all $U' \subseteq U$ and $V' \subseteq V$ with $|U'| \ge \eps |U|$ and $|V'|
	\ge \eps |V|$.

	The next lemma states that large induced graphs of regular subgraphs are still regular. The proof is straightforward by checking the definition.

    \begin{lemma}\label{lemma:slicing}
		Let $p\in \interval[open left]{0}{1}$ and $0 < \eps < \mu < 1/2$. If $G[U,V]$ is a $(\eps,p)$-regular bipartite
		graph, then for every $W \subseteq U$, with $|W| \geq \mu|U|$, the graph $G[W,V]$ is
		$(\eps/\mu,p)$-regular. Furthermore, we have $d(W,V)\geq d(U,V)-\eps p$.
	\end{lemma}
	
	The next lemma states that regular bipartite graphs contain regular subgraphs with any given (sufficiently large) number of edges. A proof can be found in~\cite[Lemma 4.3]{GeSt}.

    \begin{lemma}\label{lemma:edgereg}
		For $\eps \in (0,1/6)$, there exists $C=C(\eps) > 0$ such that the following holds. Let $G=G[V_1, V_2]$ be a $(\eps,d_{G})$-regular bipartite graph, where $d_{G}=d_{G}(V_{1},V_{2})$. 
		For all $Cn\leq m\leq e(G)$, there exists a $(2\eps,d_{H})$-regular subgraph $H=H[V_1, V_2]$ of $G$ such that $e(H)=m$, where $d_H=e(H)/|V_1||V_2|$.
	\end{lemma}
	
	The following lemma can be found in~\cite[Lemma 6]{KKM18}. It states that upper
	uniform bipartite graphs contain a bipartite subgraph which is regular and has
	the same density.
	
	\begin{lemma}\label{tools:lemma2}
		For $\eps \in (0, 1/2)$ and $\alpha \in (0,1)$, there exists $\mu>0$ such that the following holds for all $p \in (0, 1]$.
		Let $G=G[V_{1},V_{2}]$ be a $(\mu,p)$-upper uniform bipartite graph, with $|V_1| = |V_2|$ and $d(V_{1},V_{2}) \geq \alpha p$.
		There exist
		$V'_{1} \subseteq V_{1}$ and $V'_{2} \subseteq V_{2}$, with $|V_{1}'|,|V_{2}'|\ge \mu |V_{1}|$, such that $G[V'_{1},V'_{2}]$ is $(\eps,p)$-regular with density at least $\alpha p$.
	\end{lemma}
	
	Let $\eps$ be any positive real number.
	We say that a graph $G$ has the \emph{discrepancy property}
	$\mathrm{DISC}(\eps)$ if for any subsets $U,V \subseteq V(G)$, we
	have \[\left|e_G(U,V) - \frac{\mathrm{vol}(U)\;
		\mathrm{vol}(V)}{\mathrm{vol}(V(G))}\right| \leq \eps \;
	\mathrm{vol}(V(G)),\] \smallskip where $\mathrm{vol}(X) := \sum_{x \in
		X}\deg_G(x)$ for any $X \subseteq V(G)$.
	Roughly speaking, if a graph $G$ has the $\mathrm{DISC}(\eps)$ property, then its edges are almost uniformly distributed. 
	The next lemma builds a bridge between discrepancy and classical regularity
	(see~\cite{KKM18}, Lemma 4).
	
	\begin{lemma}\label{lemma:disc}
		For every $\eps, \mu > 0$ there exist $\delta,\nu >0$ such that 
		the following holds.
		Let $p \in (0, 1]$ and $G$ be an $n$-vertex graph which satisfies

		\begin{enumerate}[label=(\arabic*),itemsep=5pt]
			\item the discrepancy property $\mathrm{DISC}(\delta)$;
			\item $|\deg(v) - pn| \le \nu pn$ for $v \in V(G)$.
		\end{enumerate}
		Then, for any disjoint subsets $U,V \subseteq V(G)$ such that $|U|,|V| \ge \mu n$, the graph $G[U,V]$ is $(\eps,p)$-regular.
	\end{lemma}
	
	We end this section with a probabilistic tool.
	We state Azuma's inequality as it was stated by McDiarmid~\cite{Di}.
	Below, the notation $[M]$ stands for $\{1,2,\ldots,M\}$.
	
	\begin{lemma}\label{lemma:azuma}
		Let $X_1, \ldots, X_M$ be independent random variables, with $X_i$ taking values on a finite set
		$A_i$ for each $i \in [M]$. Suppose that $f:\prod_{i=1}^{M} A_i \rightarrow \mathbb{R}$ satisfies
		$|f(x) - f(x')| \leq c_i$ whenever the vectors $x$ and $x'$ differ only in the $i$th coordinate.
		If $Y$ is the random variable given by $Y = f(X_1, \ldots, X_M)$, then, for any $a >0$,

        \begin{align*}
			\P(|Y - \E(Y)| > a) \leq 2 \exp\left\{- \frac{2a^2}{\sum_{i=1}^M {c_i}^2}\right\}.
		\end{align*}
	\end{lemma}

\section{Pseudo-randomness and isolated copies}
\label{sec:colorpartition}

Recall from Section \ref{sec:overview} that we will consider a random partition of the colors of a proper coloring of $G(n,p)$. In this section we focus on exploring properties of the graphs induced by each color class of this partition. In particular, the proof of Theorem~\ref{thm:main} converges to an application of an embedding lemma in a sparse setting: the formerly-called K\L R conjecture, proved by Balogh, Morris and Samotij~\cite{BMS}. Therefore, the goal is to guarantee that the graphs that we consider fit in the requirements of that result. In order to state it, we need a little explanation.

Let $m$ and $n$ be positive integers with $m \leq n^2$ and let $\varepsilon>0$ and
$p \in [0,1]$. For a graph $H$ with $V(H) = [h]$, we denote by
$\G(H,n,m,p,\varepsilon)$ the family of graphs obtained in the following way. Consider
$h$ disjoint sets $V_1,\dots,V_h$, each of size $n$, and for each $ij\in E(H)$, add
$m$ edges between the pair $(V_i,V_j)$ such that the resulting bipartite graph
is $(\varepsilon,p)$-regular. We denote by $\G(H,n,m,p,\varepsilon)$ the collection of
all graphs obtained in this way. We say that a copy of $H$ in $G \in \G(H,n,m,p,\varepsilon)$  is
\emph{transversal} if the vertex $i$ is mapped to $V_i$ (we omit the dependence on the order of the $V_i$'s in the definition of $\mathcal{G}$). 
We denote by $\Gs(H,n,m,p,\varepsilon)$ the set of all graphs $G \in\G(H,n,m,p,\varepsilon)$ that do not contain a transversal copy of $H$. 

Now we are ready to state the K{\L}R conjecture.

\begin{thm}\label{thm:KLR}
	For every graph $H$ and every positive $\gamma$, there exist positive constants $B$, $n_0$ and $\varepsilon$ such that the following holds. For every $n\in \mathbb{N}$ with $n\geq n_0$ and $m\in \mathbb{N}$ with $m\geq B n^{2-1/m_2(H)}$,

	\begin{align*}
		\left| \Gs(H,n,m,m/n^2,\varepsilon) \right| \leq \gamma^m \binom{n^2}{m}^{e(H)}.
	\end{align*}
\end{thm}

\subsection{Random partition of the colors}\label{subsec:randcolor}\hspace*{\fill}

In this subsection, we prove that, in a typical outcome of $G(n,p)$, the hypothesis of Lemma~\ref{lemma:disc} is met by the graphs induced by the colors assigned to each edge of $H$.
Let $G$ be a graph, let $c: E(G) \to \mathbb{N}$ be a proper edge-coloring of $G$ and let $T$ be a positive
integer. To each color $i \in \mathbb{N}$, we assign to $i$ an element $\sigma(i) \in [T]$ chosen
uniformly at random from $[T]$ and denote by $\mathbb{P}_T$ this distribution. For each $t \in [T]$, let $G_{t}$ be the spanning subgraph of $G$
with edge set

\begin{align*}
	E(G_t) = \{ e\in E(G): \sigma(c(e)) = t \}.
\end{align*}
That is, $E(G_t)$ corresponds to the edges of $G$ for which their color was assigned to $t$.

In our problem, the graph $G$ is an outcome of $G(n,p)$. This makes us consider two probability spaces when dealing with $G_t$: the one which defines $\gnp$
and the one which defines the random assignment $\sigma$ for a fixed proper edge-coloring of $\gnp$. In order to avoid confusion, we use $\mathbb{P}$ and $\mathbb{E}$ to refer to the distribution of $\gnp$.
For a fixed proper-coloring of $c: E(G) \to \mathbb{N}$, we use $\mathbb{P}_T$ and $\mathbb{E}_T$ to refer to the distribution of the random assignment of colors.
We observe that if $\A$ is an event depending on $G_t$, for some $t$, then $\mathbb{P}_T(\A)$ is a random variable in $G(n,p)$. 
In the lemmas, $o(1)$ denotes a function tending to $0$ as $n \to +\infty$. 
Whenever we say that an event occurs `with high probability', we mean with probability $1-o(1)$.

Our aim in this subsection is to show that $G_t$ satisfies the two requirements of Lemma~\ref{lemma:disc}, which are the concentration of degrees and $\mathrm{DISC}(\varepsilon)$. We omit the proof of the first, since it follows by a straightforward Chernoff's bound argument, together with the fact that edges touching each vertex have distinct colors. The proof that $G_t$ satisfies $\mathrm{DISC}(\varepsilon)$ is more intricate and relies on results of Chung and Graham~\cite{ChGr} on pseudo-random graphs.

\begin{lemma}\label{lemma:concdeg}

	Let $\delta > 0$ and $T$ be a positive integer. If $p \gg \log{n}/n$, then the following holds for $G
	= G(n,p)$ with high probability. For any proper edge-coloring $c:E(G) \to \mathbb{N}$ of $G$ we have

    \begin{align*}
		\mathbb{P}_T\left( d_{G_t}(v) = (1 \pm  \delta) \frac{pn}{T}\right)=1-o(1),
	\end{align*}
	for every $t \in [T]$ and $v \in V(G)$.
\end{lemma}

\begin{lemma}\label{lemma:discGnp}
	Let $\eps,\beta \in (0,1)$ and $T$ be a positive integer. If $p\gg n^{-\beta}$, then the following holds for $G= G(n,p)$ with high probability. For any proper edge-coloring $c:E(G) \to \mathbb{N}$ of $G$ we have
	\[\mathbb{P}_T \big(G_t \text{ satisfies }\mathrm{DISC}(\varepsilon)\big)=1-o(1),\]
	\noindent for every $t\in [T]$.
\end{lemma}

A straightforward proof that a random graph satisfies $\mathrm{DISC}(\varepsilon)$ can be tricky, since any concentration inequality we obtain has to be stronger than the choices of subsets of the vertex set. Luckily for us, the works of Chung and Graham~\cite{ChGr} relate this property with the distribution of circuits of even length. Given a graph $G$, we say that a sequence $C = (v_1,\ldots,v_{\ell})$ of vertices of $G$ is an
\emph{$\ell$-circuit} if $v_iv_{i+1} \in E(G)$, for every $i \in [\ell-1]$, and $v_1v_{\ell} \in E(G)$. The \emph{weight} of an $\ell$-circuit $C = (v_1,\ldots,v_{\ell})$ is given by

\begin{align*}
	w(C) = \prod_{i=1}^{\ell} \frac{1}{d_G(v_i)}.
\end{align*}
We denote by $\mathcal{C}_{\ell}(G)$ the collection of all $\ell$-circuits of $G$. We say that $G$
has the $\mathrm{CIRCUIT}_{\ell}(\eps)$ property if

\begin{align*}
	\sum_{C \in \mathcal{C}_{\ell}(G)} w(C) = 1 \pm \eps.
\end{align*}

The following lemma from~\cite{ChGr} shows that $\mathrm{CIRCUIT}$ essentially implies $\mathrm{DISC}$.
\begin{lemma}\label{lemma:ChGr}
	For every $\eps>0$ and positive integer $\ell$, if $G$ has the $\mathrm{CIRCUIT}_{2\ell}(\eps)$
	property, then $G$ has the $\mathrm{DISC}(\eps^{1/2\ell})$ property.
\end{lemma}

Since the degrees in $G_t$ are concentrated around $pn/T$ for every $t\in[T]$, we basically have to show that the number of circuit of some even length $\ell$ is close to $(pn/T)^\ell$ in $G_t$. In principle, it is not clear even how to compute the expectation of this value, since the edges are not selected independently. We simplify this problem in two steps. First we call upon the result stated below and proved in \cite[Corollary 4.9]{KKM14}, which shows that for certain values of $p$, almost all $\ell$-circuits are actually cycles. For that, we denote by $\mathcal{C}'_\ell (G)$ the number of $\ell$-cycles in a graph $G$.

\begin{lemma}\label{circuit}
	Let $\ell\geq 2$ and $\delta>0$. If $p\gg n^{-1+1/\ell}$, then with high probability $G=\gnp$ satisfies
	\[|\mathcal{C}_{2\ell}(G)|\leq (1+\delta)|\mathcal{C}'_{2\ell}(G)| .\]
\end{lemma}

 Now our aim is to show that almost all cycles in a proper edge-coloring of $G(n,p)$ are in fact rainbow. If we assume that to be true, it is easy to see that the expected number of $\ell$-cycles in each $G_t$ is roughly $(pn/T)^{\ell}$, since each color is independently assigned to a class. In order to prove such statement we count the number of non-rainbow cycles basically by counting the shortest path whose first and last edges have the same color and then by completing them into cycles. Therefore, the following special case of a classical result of Spencer~\cite[Theorem 2]{Spencer} is fairly convenient.

\begin{lemma}\label{Spencerzinho}
	Let $\ell\geq 2$. If $p^\ell n^{\ell-1}\gg \log n$, then with high probability $G=\gnp$ satisfies the following. For every pair of vertices $u,v\in V(G)$, the number of $\ell$-vertex paths connecting $u$ to $v$ is $\Theta(p^{\ell-1}n^{\ell-2})$.
\end{lemma}

We remark that the values of $p$ needed to apply Lemma \ref{Spencerzinho} are lower for longer paths. This fact plays an important role in the proof of Lemma \ref{lemma:rainbowcycles}, which we are now ready to state.

\begin{lemma}\label{lemma:rainbowcycles}
	
	For $\beta \in(0,1)$, let $p\geq n^{-\beta}$ and let $\ell$ be an integer satisfying $p^{\lceil \ell/2 \rceil}n^{\lceil \ell/2 \rceil-1}\gg \log n$. With high probability, in every proper edge-coloring of $G=\gnp$ there are $O(p^{\ell-1}n^{\ell-1})$ non-rainbow $\ell$-cycles.
\end{lemma}

\begin{proof}
  Let $G=G(n,p)$ be as in the statement. Notice that as an straightforward application of Chernoff's inequality (which we will omit the details here), it follows that with high probability we have $d(v)\leq 2pn$, for every $v\in V(G)$. Fix now a proper edge-coloring of $G$.

 We say that a path in $G$ is \textit{color-tied} if the first and last edges have the same color. Note that every non-rainbow $\ell$-cycle must contain a color-tied path with at most $\lfloor \ell/2 \rfloor +2$ vertices, by considering the shortest path between edges with the same color. We then count the number of non-rainbow $\ell$-cycles by counting the number of such paths and then by counting in how many ways these paths can be extended into an $\ell$-cycle in $G$.

	To count the number of color-tied paths with $k$ vertices, for a fixed $k\in [\lfloor \ell/2 \rfloor +2]$, we first choose an ordered pair $(v_1,v_2)$ such that $\{v_1,v_2\}\in E(G)$. We then count the number of paths with $k-1$ vertices such that $(v_1,v_2)$ are respectively the first and second vertices by inductively extending the path choosing vertices from the neighborhood of the last vertex of the current path. Finally, there is at most one way of extending the current path with $k-1$ vertices to a color-tied path with $k$ vertices. This is because the last edge must have the same color as $\{v_1,v_2\}$ and since the coloring is proper, there must be at most one neighbor of the last vertex of the current path with that color. Therefore, the number of color-tied $k$-paths is at most
	\[
      4pn^2\cdot (2pn)^{k-3}\cdot 1=O(p^{k-2}n^{k-1}),
    \]
which is smaller than the number of $k$-paths by a factor of $\Omega(pn)$.

Now let $v_1v_2\dots v_k$ be a color-tied path, with $k\in [\lfloor \ell/2 \rfloor +2]$. Since
    \[
      p^{\ell-k+2}n^{\ell-k+1}
      \gg p^{\lceil \ell/2 \rceil}n^{\lceil \ell/2 \rceil-1}
      \gg \log n,
    \]
	\noindent then with high probability the number of $(\ell-k+2)$-paths connecting $v_1$ to $v_k$ is of order $p^{\ell-k+1}n^{\ell-k}$, by Lemma~\ref{Spencerzinho}. Therefore, the number of non-rainbow cycles is at most
	\[
      \sum_{k=1}^{\lfloor \ell/2 \rfloor +2}
      O\left(p^{k-2}n^{k-1}\cdot p^{\ell-k+1}n^{\ell-k}\right)
      = O(p^{\ell-1}n^{\ell-1}).
    \]
	
\end{proof}

Now we put all pieces together to prove Lemma~\ref{lemma:discGnp}.

\begin{proof}[Proof of Lemma~\ref{lemma:discGnp}]
  Let $\varepsilon,\beta\in (0,1)$ and let $G = G(n,p)$ with $p=n^{-\beta}$ and $c:E(G)\rightarrow \mathbb{N}$ a proper edge-coloring. We choose an integer $\ell$ such that $\beta<1-1/(2\ell-1)$ and we set $\varepsilon'=\varepsilon^{2\ell}$. Our goal is to show that with high probability we have for $G=G(n,p)$ that
  \[
    \P_T\left(G_t\text{ satisfies }\mathrm{DISC}(\varepsilon) \right)=1-o(1),
  \]

  \noindent for every $t\in [T]$, in a random assignment of colors of $c$ into $T$ classes. For any $\delta >0$ with high probability we have that $\mathbb{P}_T(d_{G_t}(v) = (1 \pm 3\delta)pn/T)=1-o(1)$ for all $v \in  V(G)$ and all $t \in [T]$, by Lemma~\ref{lemma:concdeg}. Therefore, by choosing $\delta$ small enough as a function of $\eps'$, we have

  \begin{align}\label{pikachu}
    \sum_{C \in \mathcal{C}_{2\ell}(G_t)} w(C)
    = \sum_{C \in \mathcal{C}_{2\ell}(G_t)} \prod_{v \in V(C)} \frac{1}{d_{G_t}(v)} \;
    = \;|\mathcal{C}_{2\ell}(G_t)|\left( 1 \pm \frac{\eps'}{3} \right) {\left(\frac{T}{pn}\right)}^{2\ell},
  \end{align}
where $\mathcal{C}_{2\ell} (G_t)$ denotes the set of $2\ell$-circuits in $G_t$.

  If we show that, under the distribution $\mathbb{P}_T$, $G_t$ has the $\mathrm{CIRCUIT}_{2\ell}(\varepsilon')$ property with probability $1-o(1)$, then Lemma~\ref{lemma:ChGr} implies that the same holds for $\mathrm{DISC}(\varepsilon)$. By the definition of the $\mathrm{CIRCUIT}_{2\ell}(\varepsilon')$ property, we have to show that $\sum w(C)=1\pm \varepsilon'$. By \eqref{pikachu}, it is sufficient to prove that

  \begin{align}\label{cycleconc}
    \mathbb{P}_T\left(|\mathcal{C}_{2\ell}(G_t)|
    = \left( 1 \pm \frac{\eps'}{3} \right) {\left(\frac{pn}{T}\right)}^{2\ell}\right)
    = 1 - o(1).
  \end{align}

  Recall that $\mathcal{C}'_{2\ell}(G_t)\subset \mathcal{C}_{2\ell}(G_t)$ is the collection of $2\ell$-cycles in $G_t$. Since $p\gg n^{-1+1/2\ell}$, the numbers of $2\ell$-circuits and of $2\ell$-cycles in $G(n,p)$ are with high probability asymptotically equal. More precisely, by choosing some small $\delta$ in Lemma \ref{circuit} we infer that with high probability $G$  satisfies

  \begin{align}\label{blastoise}
  |\mathcal{C}_{2\ell}(G)\setminus \mathcal{C}'_{2\ell}(G)|
  \leq \dfrac{\eps'}{6}\left(\dfrac{pn}{T}\right)^{2\ell}.
  \end{align}

  Therefore, we can remove the $2\ell$-circuits that are not $2\ell$-cycles from the computation. In doing that, we reduce the proof of  \eqref{cycleconc} to proving that

\begin{align}\label{charizard}
\mathbb{P}_T\left(|\mathcal{C}'_{2\ell}(G_t)|= \left(1\pm \dfrac{\varepsilon'}{6}\right)\left(\dfrac{pn}{T}\right)^{2\ell}\right)=1-o(1),
\end{align}
for every $t\in [T]$.

  In order to prove \ref{charizard}, fix $t\in [T]$ and for each $i \in c(E(G))$, let $A_i = \{0,1\}$ and let $X_i$ be the
  indicator function for the event $\{\sigma(i) = t\}$ and set $Y = |\mathcal{C}'_{2\ell}(G_t)|$. Note
  that $Y=f(X_1,\ldots,X_{r})$, for some $f:\prod_{i=1}^{r}A_i \rightarrow \mathbb{R}$. Now,
  let $c_i$ be the smallest real number for which $|f(x)-f(x')| \leq c_i$, whenever \(x,x' \in \prod_{i=1}^{r} A_i\) differ
  only on the $i$th coordinate. By double counting the pairs $(i,e)$ such that $i\in c(E(G))$ and $e\in E(G)$ has color $i$ and is contained in a $2\ell$-cycle, we obtain

  \begin{align}\label{eq:sumci}
    \sum_{i=1}^{r}c_i \leq 2\ell |\mathcal{C}'_{\ell}(G)|.
  \end{align}

  Moreover, since $pn^{1-2/(2\ell-1)} \to \infty$, Lemma~\ref{Spencerzinho} implies that the number of of $2\ell$-cycles in $G$ containing a given edge $e \in G$ is at most $Dp^{2\ell-1}n^{2\ell - 2}$, for some large constant $D>0$. Since each color $i \in [r]$ induces a matching in $G$, it follows that $c_i= Dp^{2\ell-1}n^{2\ell - 1}$. Together with \eqref{eq:sumci}, we obtain that

  \begin{align*}
    \sum_{i = 1}^{r} c_i^2
    \leq Dp^{2\ell-1}n^{2\ell - 1}\sum_{i = 1}^{r} c_i
    \leq 2\ell D p^{2\ell-1}n^{2\ell - 1}|\mathcal{C}'_{2\ell}(G)|.
  \end{align*}

  To finish the proof of \ref{charizard}, we have to calculate the expectation of $Y$. Since each edge is in $G_t$ with probability $1/T$, then for each $C\in \mathcal{C}'_{2\ell}(G)$, we have that $\mathbb{P}_T (C\in G_t)\geq 1/T^{2\ell}$, with equality being attained if $C$ is rainbow. Since $(pn)^{2\ell}\gg 1$, the number of $2\ell$-cycles in $G(n,p)$ is with high probability $(1-o(1))(pn)^{2\ell}$. By Lemma~\ref{lemma:rainbowcycles}, since $p^{\ell-3}n^{\ell-4}\gg \log n$, almost all of those $2\ell$-cycles are rainbow. Therefore, with high probability (with respect to $G(n,p)$) we have that
  \[\mathbb{E}_T(Y)=\left(1\pm \dfrac{\varepsilon'}{12}\right)\left(\dfrac{pn}{T}\right)^{2\ell}.\]

  Finally, by Lemma~\ref{lemma:azuma},

  \begin{align*}
    \mathbb{P}_T\left(|Y - \E[Y]| >  \dfrac{\varepsilon'}{12}\left(\dfrac{pn}{T}\right)^{2\ell} \right)
    & \leq 2\exp\left\{- \frac{\Omega((pn)^{4\ell})}{\sum_i^{r} {c_i}^2}\right\} \\
    &=2\exp{-\Omega(pn)} \\
    & = o(1).
  \end{align*}
  Therefore, with probability tending to $1$ (under $\mathbb{P}_T$) equation \eqref{charizard} holds, which, together with with \eqref{blastoise}, implies \eqref{cycleconc} and finishes the proof.

\end{proof}

\subsection{Isolated copies and their distribution}\label{subsec:isolated}\hspace*{\fill}

Given a graph $G$ on $n$ vertices, we call a copy of a graph $S$ in $G$ \emph{isolated} if it does not share an edge with any other copy of $S$ in $G$. Let $G^S$ be the spanning subgraph of
$G$ induced by all the edges that belong to isolated copies of $S$ in $G$. Since $G^S$ will play a role in the application of Theorem \ref{thm:KLR}, the following issue arises. Theorem \ref{thm:KLR} is a counting result that states that there exists only a very small number of graphs that possess some pseudo-random properties and do not copies of a graph $H$. However, since $G^S$ is not $G(n,p)$ then it is not clear how to bound the probability of appearance of such graphs, which motivates the lemma below. For $E\subseteq E(K_n)$,
we write $E \sqsubseteq G^{S}$ if $E\subseteq G^S$ and if additionally no two edges in $E$ belong to
the same isolated copy of $S$.

\begin{lemma}\label{lemma:negcorrelation}
	Let $F$ be a graph and $G= \gnp$, with $p = p(n) \in \interval[open left]{0}{1}$. If
	$q=n^{v(F)-2}p^{e(F)}$, then for any $E \subseteq E(K_n)$ we have that

    \begin{align*}
		\P[E \sqsubseteq E(G^{S})] \leq q^{|E|}.
	\end{align*}
\end{lemma}

We move to other properties of $G^S$, namely the density and upper-uniformity, to be able to apply Lemma~\ref{tools:lemma2} and hence the Theorem \ref{thm:KLR}. For disjoint sets $V_1,\ldots,V_{v(S)} \subseteq V(G)$, we
denote by $Z_G(V_1,\ldots,V_{v(S)})$ the number of \emph{transversal} copies of $S$ in
$G[V_1,\ldots,V_{v(S)}]$, i.e., copies of $S$ in $G$ with one vertex in each $V_i$ with $i \in
[v(S)]$ and by $Y_G(V_1,\ldots,V_{v(S)})$ the number of such copies that are also isolated. We may omit the sets
$V_1,\ldots,V_{v(S)}$ from the notation if they are clear from the context. For $G = \gnp$ and for disjoint linear-sized sets $V_1,\ldots,V_{v(S)} \subseteq V(G)$, we have that

\begin{align*}
	\E[Z(V_1,\ldots,V_{e(S)})] = \Theta\left( n^{v(S)}p^{e(S)} \right).
\end{align*}

Our analysis of $G_t^S$ starts by proving that in a typical outcome of $G=G(n,p)$ and in a typical assignment of colors we have $Y_{G_t}(V_1,\dots,V_{v(S)})=\Omega(n^{v(S)}p^{e(S)})$. In words, we mean that the number of isolated copies of $S$ in each $G_t$, with $t\in [T]$, is a considerable proportion of all copies of $S$ in $G$. Therefore, the bipartite graph $G^S_t[V_i,V_j]$, for $ij \in E(S)$, has $\Omega(n^{v(S)}p^{e(S)})$ many edges because the copies are isolated. Under the hypothesis of Theorem~\ref{thm:main}, this lower bound is larger than $\Omega(n^{2-1/m_2(H)})$, which is one of the requirements for applying Theorem~\ref{thm:KLR} successfully.

\begin{lemma}\label{prop:canspan}
	Let $S$ be a graph. For every $\mu>0$ and integer $T>0$ there exists $\alpha>0$ such that for $p=n^{-\beta}$ and $1/m_2(S)<\beta \leq (v(S)-1)/e(S))$ the following holds.  With high probability for every proper edge-coloring of $G=G(n,p)$ and for a fixed family of disjoint sets
	$V_1, \ldots, V_{v(S)} \subseteq V(G)$ with $|V_i| \geq \mu n$, for $i \in [v(S)]$, we have

    \begin{align*}
		\mathbb{P}_T\left(Y_{G_t} \geq \alpha n^{v(S)}p^{e(S)}\right)=1-o(1),
	\end{align*}
	for every $t \in [T]$.
\end{lemma}

Another requirement for our proof is that $G^S$, and hence $G^S_t$, is $(\mu,q)$-upper uniform for $q= O(n^{v(S)}p^{e(S)})$, which is stated in the next lemma. The proof follows by an application of Lemma~\ref{lemma:negcorrelation} and it was originally stated in \cite[Lemma 14]{KK}.

\begin{lemma}\label{lemma:upperuniff}
	Let $S$ be a graph and $\beta >0$ be such that $\beta < (v(S)-1)/e(S)$. Let $p = Bn^{ -
		\beta}$, with $B>0$, and $G = \gnp$. Then for every $\mu>0$ with high probability we have that $G^S$ is
	$(\mu,q)$-upper-uniform, where $q = 6e(S)n^{v(S)-2}p^{e(S)}$.
\end{lemma}

The rest of this section is dedicated to the proof of Lemma~\ref{prop:canspan}. First proving that for some regimes of $p$ and for a fixed family of pairwise disjoint sets $V_1,\dots, V_{v(S)}$ with high probability we have that a big proportion of the transversal copies are actually isolated.

\begin{prop}\label{prop:gnpIsol}
	Let $S$ be a 2-balanced graph. If $p=n^{-\beta}$ and $\frac{1}{m_2(S)}<\beta \leq \frac{(v(S)-1)}{e(S)}$, then with high probability $G=G(n,p)$ satisfies the following. For
	every $\mu>0$ and every family of pairwise disjoint sets $V_1, \ldots, V_{v(S)} \subseteq V(G)$,
	with $|V_i| \geq \mu n$, we have
    
	\begin{align*}
		\E[Z(V_1,\ldots,V_{e(S)})] = \Omega\left( n^{v(S)}p^{e(S)} \right).
	\end{align*}
\end{prop}

\begin{proof}
	Let $X$ be the number of copies of any graph that is a union of two copies of $S$ intersecting in at least one edge and which is not $S$. Let $V_1,\dots V_{v(S)}\subset V(G)$ be a fixed family of pairwise disjoint sets. The proof lies basically on showing that

	\begin{equation}\label{Ex}
		\Delta
        := \sum_{K_2\subset J \subset S}
        O\left(n^{2v(S)-v(J)}p^{2e(S)-e(J)}\right)
        = o(\mathbb{E}(Z)).
	\end{equation}
	
	Indeed, since $S$ is 2-balanced we have that
	\[
      \frac{v(S)-v(J)}{e(S)-e(J)} \leq \frac{1}{m_2(S)}
      \implies n^{v(S)-v(J)}p^{e(S)-e(J)} = o(1),
    \]
	for every $K_2\subset J\subset S$, which proves \eqref{Ex} since $\mathbb{E}Z=\Omega(n^{v(S)}p^{e(S)})$. By Janson's inequality~\cite{janson} and by the fact that $\Delta<\mathbb{E}Z$, we have that
	\[
      \mathbb{P}\left(Z\leq 3\mathbb{E}Z/4\right)
      = \exp (-\Omega(\mathbb{E}Z)).
    \]
	
	Since $\mathbb{E}Z=\Omega(n^{v(S)}p^{e(S)})\gg n$, by taking union bound $Z\geq 3\mathbb{Z}/4$ for every choice of $V_1,\dots, V_{v(S)}$.

		 Moreover, since $\mathbb{E}X= O(\Delta)$, then by Markov's inequality we have that with high probability $X=o(\mathbb{E}Z)$ and, therefore, we can discount every copy of $S$ which shares at least an edge with other copy, to finish the proof.
\end{proof}

We use Proposition~\ref{prop:gnpIsol} together with Azuma's inequality to prove Lemma~\ref{prop:canspan} and finish this section.

\begin{proof}[Proof of Lemma~\ref{prop:canspan}]
	Let $G=G(n,p)$, with $p=n^{-\beta}$ and $1/m_2(S)<\beta \leq (v(S)-1)/e(S))$, and let $c:E(G)\rightarrow [r]$ be a proper edge-coloring of $G$, for some $r \in \mathbb{N}$. For an integer $T>0$ consider a random partition of the colors into $T$ classes. For each $i \in [r]$, let $X_i$ be the indicator function for the event $\sigma(i) =
	t$ and observe that $Y_{G_t}$ is a function of $X_1, \ldots, X_{r}$. For each $i\in [r]$, let
	$c_i = c_i(G)$ be the smallest real number such that if we change the value of $X_i$ only, then
	the value of $Y_{G_t}$ will be altered by at most $c_i$. Since the coloring of $G$ is proper, by
	altering the value of $X_i$ we add or remove at most a perfect matching from $G_t$, which implies
	that it will affect at most $n$ isolated copies of $S$. Therefore, we have $c_i \leq n$.
	Furthermore, since a transversal isolated copy of $S$ in $G$ can be affected by at most $e(S)$
	changes in the value of $X_1,\ldots,X_{r}$, we also have

    \begin{align*}
		\sum_{i=1}^{r} c_i
		\leq e(S) Y_G.
	\end{align*}
	Hence,

    \begin{align*}
		\sum_{i=1}^{r} c_i^2
		\leq n \sum_{i=1}^{r} c_i
		\leq e(S) n Y_G.
	\end{align*}
	Furthermore, notice that each copy of $S$ in $G$ belongs to $G_t$ with probability $(1/T)^k$, where $k$ is the number of colors that appears in such copy of $S$. In particular, such copy of $S$ is in $G_t$ with probability at least $(1/T)^{e(S)}$. Therefore,

    \begin{align*}
		\E_T[Y_{G_t}]
		\geq \frac{Y_G}{T^{e(S)}}.
	\end{align*}
	
	Note that, by Proposition~\ref{prop:gnpIsol}, we have that $Y_G =
	\Omega(n^{v(S)}p^{e(S)})$. Therefore, Lemma~\ref{lemma:azuma} yields that

    \begin{align*}
		\P_T\left[|Y_{G_t} - \E[Y_{G_t}]| \geq \frac{1}{2}\E[Y_{G_t}]\right]
		&\leq 2 \exp \left\{ - \frac{{\E[Y_{G_t}]}^2}{2\sum_{i \in [q]} {c_i}^2} \right\}\\
		&\leq 2 \exp \left\{ - \frac{Y_G^2}{T^{2e(S)}e(S) n Y_G} \right\}\\
		&= \exp\left\{- \Omega(n^{v(S) -1}p^{e(S)})\right\}\\
		&= o(1).
	\end{align*}
	The last line is due to the fact that $\beta < (v(S)-1)/e(S)$. Consequently, with probability $1-o(1)$ under $\mathbb{P}_T$, we have that

    \begin{align*}
		Y_{G_t}
		\geq \frac{1}{2} \E[Y_{G_t}]
		\geq \frac{Y_G}{2T^{e(S)}}
		\geq \alpha n^{v(S)}p^{e(S)},
	\end{align*}
	for some $\alpha > 0$ that only depends on the graph $S$ and the values of $T$ and $\mu$.

\end{proof}

\section{Proof of Theorem~\ref{thm:main}}\label{section:main}

Let $H$, $F$ and $S$ be as in the statement of Theorem~\ref{thm:main}. 
Let us say that $V(H) =\{u_1, \ldots, u_h\}$, $V(S) = \{v_1, \ldots, v_{s}\}$ and $V(F) = \{w_1, \ldots, w_f\}$, where $u_1 = v_1$ and $u_2 = v_2$.
We consider an equi-partition of $[n]$ into $ s + h - 2$ sets

\begin{eqnarray*}
	V = \left( \bigcup_{i=1}^{s} V_i\right) \cup \left( \bigcup_{i=3}^{h} U_i\right).
\end{eqnarray*}

For technical reasons related to the Theorem~\ref{thm:KLR}, instead of working in $G(n,p)$, we work on the random graph $G$ obtained in the following way. Each pair of vertices in $V$ contained in $\cup_{i=1}^{s}V_i$ forms an edge with probability $p$ and every pair of vertices in $V$ forms an edge with probability $q'=e(H)q$, where $q := 6e(S)n^{v(S)-2}p^{e(S)}$, all independently from each other. Notice that $q' < p$, and therefore, since we are dealing with a monotone property, in order to prove Theorem~\ref{thm:main}, it suffices to prove that with high probability the random graph $G$ satisfies $G\flecha F \oplus H$.

Let $c:E(G)\rightarrow \mathbb{N}$ be a proper coloring of $E(G)$. 
For each color $i \in e(E(G))$, we assign
independently and uniformly at random an edge $\sigma(i) \in E(H)$. For each $u_i u_j \in E(H)$, let
$G_{u_i u_j}$ be the spanning subgraph of $G$ with edge set

\begin{align*}
	E(G_{u_i u_j}) = \{e \in E(G): \sigma(c(e)) = u_i u_j\}.
\end{align*}
That is, $E(G_{u_i u_j})$ are those edges of $E(G)$ for which their color was assigned to $u_i u_j$.

Recall that $Y_{G_{u_1u_2}}(V_1, \ldots, V_s)$ denotes the number of  transversal isolated copies of $S$ in $G_{u_1u_2}[V_1, \ldots, V_{s}]$ (c.f. Section~\ref{subsec:isolated}).
By Lemma~\ref{prop:canspan}, there exists $\alpha>0$ such that the following holds with high probability.
For every proper edge-coloring $c$ of $G$ we have

\begin{align}\label{isolado}
	\P_{e(H)}\big(Y_{G_{u_1u_2}} \ge \alpha n^{v(S)}p^{e(S)}\big) = 1-o(1). 
\end{align}
From now on, we assume that $\alpha \in (0, s^{-2f})$.
For an outcome $G$ and a proper edge-coloring $c:E(G)\to \mathbb{N}$, we denote by $\mathcal{E}_1=\mathcal{E}_1(G,c)$ the event in which $Y_{G_{u_1u_2}} \ge \alpha n^{v(S)}p^{e(S)}$. Note that $\mathcal{E}_1$ is an event in the probability space given by $\sigma$.

Let $G^{S}(V_{[s]})$ be the spanning subgraph of $G$ induced by all the edges that belong to isolated copies of $S$ in $V_{[s]}:= V_1 \cup \cdots \cup V_s$.
By Lemma~\ref{lemma:upperuniff}, $G^{S}(V_{[s]})$ is $(\mu, q)$-upper uniform with high probability for any constant $\mu > 0$.
As we need regularity to apply Theorem \ref{thm:KLR}, from now on we fix $\eps \in (0,1/8)$ (to be chosen later) and we let $\mu=\mu(\eps,\alpha)>0$ be given by Lemma~\ref{tools:lemma2}.

Let $\cP$ be the set of graphs $J$ for which $J^{S}(V_{[s]})$ is $(\mu, q)$-upper uniform.
Our next claim states that, if $G \in \cP$ and the event $\mathcal{E}_1$ occurs, then then following holds.
There is a large and dense regular bipartite graph subgraph of $G_{u_1u_2}[V_1,V_2]$ whose edges are contained in distinct isolated rainbow copies of $F$.

\begin{claim}\label{claim:regF}
	Suppose that $G\in\cP$ and let $c: E(G)\to \mathbb{N}$ be a proper edge-coloring. If the event $\mathcal{E}_1$ occurs, then the following holds.
	For some $i_1,i_2 \in [s]$ there exist
	$W_{i_1}\subseteq V_{i_1}$, $W_{i_2}\subseteq V_{i_2}$ with $|W_{i_1}|=|W_{i_2}|\geq \mu |V_{i_1}|$
	and a bipartite graph $B_{12}\subset G_{u_1u_2}[W_{i_1},W_{i_2}]$ such that

	\begin{enumerate}[label=(\arabic*),itemsep=5pt]
		\item For every $ab \in E(B_{12})$, with $a\in W_{i_1}$ and $b\in W_{i_2}$, there is an isolated
		rainbow copy of $F$ in $V_{[s]}$ containing $ab$ whose colors are assigned to $u_1u_2$. Moreover, in these copies the vertices $a$ and $b$ correspond to the vertices $w_1$ and $w_2$, respectively.
		\item The graph $B_{12}$ is $(\eps,q)$-regular with density at least $\alpha^2 q$
	\end{enumerate}
\end{claim}

\begin{claimproof}
	Since $S \flecha F$, in each  transversal isolated copy of $S$ in $G_{u_1u_2}[V_1,\ldots,V_{s}]$
	we can find an  transversal isolated rainbow copy of $F$. Note that there are at most ${s}^f$
	different ways for a copy of $F$ to be transversal in $G_{u_1u_2}[V_1,\ldots,V_{s}]$. As $\mathcal{E}_1$ holds, by the
	pigeon-hole principle we have for some $i_1, \ldots, i_f \in [s]$ at least

    \begin{align*}
		\frac{\alpha}{{s}^f} n^{v(S)}p^{e(S)}
	\end{align*}
	transversal isolated rainbow copies of $F$ in $G_{u_1u_2}[V_{i_1},\ldots,V_{i_f}]$ with the
	corresponding copy of $w_t$ belonging to $V_{i_t}$, for each $t \in [f]$. We turn our attention to
	the bipartite graph $B_{12}=(V_{i_1}\cup V_{i_2};E')$ induced by the pairs contained in those copies of $F$.
	Observe that $B_{12}$ already satisfies property (1). Note that each edge of $B_{12}$ is in
	exactly one of the previously considered copies of $F$. Therefore, for $\alpha < s^{-2f}$ we have

    \begin{align*}
		E(B_{12})\geq \frac{\alpha}{{s}^f} n^{v(S)}p^{e(S)} =
		\frac{\alpha}{6{s}^f e(S)}qn^2 \geq  \alpha^2q|V_{i_1}||V_{i_2}|.
	\end{align*}

	As $G \in \cP$, the graph $G^{S}(V_{[s]})$ is $(\mu,q)$-upper uniform and so it is $B_{12}$.
	Moreover, as $d_{B_{12}}(V_{i_1},V_{i_2})\ge \alpha^2 q$, we can apply Lemma~\ref{tools:lemma2} to show that the following holds.
	There exists $W_{i_1}\subseteq V_{i_1}$, $W_{i_2}\subseteq
	V_{i_2}$ with $|W_{i_1}|=|W_{i_2}|\geq \mu |V_{i_2}|$, such that the bipartite graph
	$B_{12}[W_{i_1},W_{i_2}]$ is $(\eps,q)$-regular with density at least $\alpha^2q$. This shows property (2).

\end{claimproof}

From now on we assume that, if the event $\mathcal{E}_1$ occurs, then in Claim~\ref{claim:regF} we have $i_1 = 1$ and $i_2 = 2$.
We shall denote by $W_1$ and $W_2$ the sets obtained from this claim.

For a pair $ij$ such that $u_iu_j \in E(H)\setminus \{u_1u_2\}$, let $G^{ij}$ be the graph induced by $G$ on $U_i\cup U_j$.
Our next goal is, roughly speaking, to show that there exists a large and dense regular bipartite graph in $G_{u_i u_j}[U_i,U_j]$.
The density will be given by the concentration on the degrees (c.f. Lemma~\ref{lemma:concdeg}).
The regularity will be given by the degrees and the $\disc$ property (c.f. Lemma~\ref{lemma:discGnp}) combined with Lemma~\ref{lemma:disc}.

Recall that each edge in $G^{ij}$ is included independently with probability $q'=e(H)q$, where $q = 6e(S)n^{v(S)-2}p^{e(S)}$.
For $\delta >0$, an outcome $G$ and a proper edge-coloring $c:E(G)\to \mathbb{N}$, we denote by $\mathcal{E}_2=\mathcal{E}_2(G,c, \delta)$ the event in which 

\begin{align}\label{almostthere}
	d_{G_{u_iu_j}^{ij}}(v) = (1 \pm  \delta) qv(G^{ij})
\end{align}
for all $v \in U_i \cup U_j$ and pairs $ij$ such that $u_iu_j \in E(H)\setminus \{u_1u_2\}$.
Note that $\cE_2$ is an event in the probability space given by $\sigma$.
By Lemma~\ref{lemma:concdeg}, the following holds with high probability for any fixed $\delta >0$.
For every proper edge-coloring of $G^{ij}$ we have

\begin{align}\label{degree}
	\mathbb{P}_{e(H)}\left( \mathcal{E}_2 \right)=1-o(1).
\end{align}
From now on, we fix $\delta(\eps, \mu/2) \in (0,1/4)$ given by Lemma~\ref{lemma:disc}. 
Applying Lemma~\ref{lemma:concdeg} with $\mu/2$ instead of $\mu$ it is only a subtle technicality.
This choice ensures that we have a regular bipartite graph with classes of size $|W_1|$, which will be necessary to apply Theorem \ref{thm:KLR}.

Let $\eps' = \eps'(\eps,\mu/2) \in (0,\mu^22^{-5sh})$ be given by Lemma~\ref{lemma:disc}.
For an outcome $G$ and a proper edge-coloring $c:E(G)\to \mathbb{N}$,
define $\cE_3=\cE_3(G,c)$ to be the event in which $G_{u_iu_j}^{ij}$ satisfies $\mathrm{DISC}(\eps')$.
By Lemma~\ref{lemma:discGnp}, with high probability we have that for every proper edge-coloring 

\begin{align}	
	\mathbb{P}_{e(H)} \big(\cE_3 \big)=1-o(1).
\end{align}

Our next claim states that if $G \in \cP$ and the event $\cE_2 \cap \cE_3$ holds, then we can find a large and dense regular bipartite graph in $G_{u_i u_j}[U_i,U_j]$.

\begin{claim}\label{claim:Hpart}
	For a fixed outcome of $G$ and a proper edge-coloring $c$, suppose that the event $\cE_2 \cap \cE_3$ occurs. Then, 
	for $i\in[h]\setminus\{1,2\}$ there exists $W_i\subset U_i$,
	with $|W_i|=|W_1|$, such that the following holds. For every $u_i u_j\in
	E(H)\setminus\{u_1u_2\}$ the bipartite graph $G_{u_i u_j}[W_i,W_j]$ is $(\varepsilon,q)$-regular
	with density at least $\alpha^2q$.
\end{claim}

\begin{claimproof}
	Under the event $\mathcal{E}_2 \cap \mathcal{E}_3$, Lemma~\ref{lemma:disc} guarantees that the following holds.
	For $u_i u_j\in
	E(H)\setminus\{u_1u_2\}$ and any disjoint subsets $W_i\subset U_i$ and $W_j\subset U_j$ such that $|W_i|=|W_j|=|W_1|$, 
	the bipartite graph $G_{u_i u_j}[W_i,W_j]$ is $(\varepsilon,q)$-regular.
	Fix any choice for those sets $W_i$. 
	Now we are left to show that $G_{u_iu_j}[W_i,W_j]$ has density at least $\alpha^2q$.

	As $G_{u_iu_j}^{ij}$ satisfies $\mathrm{DISC}(\eps')$, we have

    \begin{align}\label{eq:boundedgesdisc}
		e(G_{u_iu_j}[W_i, W_j])
		& \geq \frac{\mathrm{vol}(W_i)\mathrm{vol}(W_j)}{\mathrm{vol}(G_{u_iu_j}^{ij})} - \varepsilon' \cdot \mathrm{vol}(G_{u_iu_j}^{ij}),
	\end{align}
	where the volume is over $G_{u_iu_j}^{ij}$.
	By simplicity, set $k:=v(G^{ij})$.
	By~\eqref{almostthere}, we have $\mathrm{vol}(G_{u_iu_j}^{ij}) < 2qk^2$. Moreover, the volumes $\mathrm{vol}(W_i)$ and $\mathrm{vol}(W_j)$ are both lower bounded by $qk|W_i|/2$.
	Therefore, it follows from~\eqref{eq:boundedgesdisc} that

    \begin{align*}
		e(G_{u_iu_j}[W_i, W_j])
		& \geq {\left(\frac{ qk}{2} \right)}^2\left(\frac{1}{2qk^2}\right)|W_i||W_j|-2\eps' qk^2\\
		& \geq \frac{q|W_i||W_j|}{16}.
	\end{align*}
	In the last inequality, we used that $\varepsilon' \le \mu^22^{-5sh}$.
	As $\alpha \le 1/8$, it follows that the density in $G_{u_i u_j}[W_i,W_j]$ is at least $\alpha^2q$.
\end{claimproof}

By Lemma~\ref{lemma:upperuniff}, with high probability we have that $G \in \cP$.
By combining this with Lemmas~\ref{lemma:concdeg},~\ref{lemma:discGnp} and~\ref{prop:canspan}, with high probability every proper edge-coloring of $G$ satisfies
\[\mathbb{P}_{e(H)}(\mathcal{E}_1\cap\mathcal{E}_2\cap \mathcal{E}_3)=1-o(1).\]
From this we infer that for a typical sample of $G$ and a fixed proper edge-coloring $c$ there must exist an assignment $\sigma'$ of colors of $c$ into $E(H)$ such that the conclusion of Claims~\ref{claim:regF} and~\ref{claim:Hpart}
hold. Therefore, in this setup we find a subgraph $G'[W_1,\ldots,W_h]$ of $G$ with the following properties:

\begin{enumerate}[label=(\alph*),itemsep=5pt]
	\item $|W_i|=n_0$ for all $i \in [h]$, for some $n_0 \ge \mu|V_1|$;
	\item  $G'[W_i,W_j] \se G_{u_i u_j}[W_i,W_j]$ is $(2\varepsilon /\mu, m/n^2)$-regular and has $m=\alpha^2qn_0^2\gg n$ edges
	for all $u_iu_j \in E(H)$;
	\item 
	For every $ab \in G'[W_1,W_2]$, with $a\in W_1$ and $b\in W_2$, there is a
	copy $F_{ab}$ of $F$ in $V_{[s]}$ containing $ab$.  Moreover, in these copies the vertices $a$ and $b$ correspond to the vertices $w_1$ and $w_2$, respectively.
	\item $F_{ab}$ is a rainbow graph  whose colors are assigned to $u_1u_2$.
	%which is $(\varepsilon,q)$-regular and has density at least $\alpha^2q$;
\end{enumerate}
%This subgraph is guaranteed to exist by Lemma~\ref{lemma:edgereg} combined with Claims~\ref{claim:regF} and~\ref{claim:Hpart}.
Observe that properties $(a)$ and $(b)$ guarantee that $G' \in \mathcal{G}(H,n_0,m,m/n_0^2, \eps)$.

Now we claim that if $G'$ contains a transversal copy of $H$, then $G$ contains a rainbow copy of $F \oplus H$ in the coloring $c$.
In fact, as the edges in $G'[W_i,W_j]$ only use colors assigned to $u_i u_j$, any transversal copy of $H$ in $ G'$ is rainbow.
Moreover, by item $(c)$, each transversal copy of $H$ in $G'$ can be extended to a copy of $F\oplus H$ in $G$, where $F$ is is rainbow and only uses colors assigned to $u_1u_2$. Therefore, we conclude that the copy of $F\oplus H$ we found is also rainbow.

Let $\gamma \in (0,1)$ be a constant to be chosen later and let $B(\gamma)$ be given by Theorem \ref{thm:KLR}.
Let $p \ge Cn^{-\beta (H, S)}$, where $C>B(\gamma)$ is another constant to be chosen later.
We observe that the constants chosen before do not depend on these parameters.
Let $\mathcal{G}^*(n)=\mathcal{G}^*$ be the family of graphs $G'$ such that $V(G')\se [n]$, the items $(a)$-$(c)$ are satisfied and no transversal copy of $H$ is contained in $G'$.
If $G \naoflecha H \oplus F$ and $G$ is a typical graph then, by the discussion above, there is a coloring $c$ and graph $G'(c) \se G$ satisfying items $(a)$-$(d)$ such that it does not contain a transversal copy of $H$.
In particular, we have

\begin{align*}	
	\P \big ( G \naoflecha F \oplus H \big ) \le \P \left( \bigcup_{G' \in \G^{*}} \big \{G' \se G \big \} \right)+o(1).
\end{align*}
By~\cref{lemma:negcorrelation} and properties $(b)$-$(c)$, we have

\begin{align*}	
	\mathbb{P}\big(G'[W_1,W_2] \sqsubseteq G^F \big) \leq q^m 
	\qquad \text{and} \qquad
	\mathbb{P}\big( G'[W_i,W_j]\subset G\big) \leq q^m
\end{align*}
for $u_i u_j\in E(H)\setminus \{u_1,u_2\}$, and hence

\begin{align}\label{almostfinal}
	\P \big ( G \naoflecha F \oplus H \big ) \le |\G^{*}|q^{e(H)m}+o(1).
\end{align}

Now we move to the application of Theorem \ref{thm:KLR}
Let $\varepsilon(\gamma) \in (0,1/2)$ given by Theorem~\ref{thm:KLR}. 
As long as $m > B n^{2-1/m_2(H)}$, we have

\begin{align*}
	\left| \G^{*} \right| \leq 2^{e(H)n}\gamma^m \binom{n_0^2}{m}^{e(H)}.
\end{align*}
The factor of $2^{e(H)n}$ accounts for the choices of the subsets $W_1,\ldots,W_h$.
To see if $m > B n^{2-1/m_2(H)}$, recall that $m=\alpha^2qn_0^2$ and $q = 6e(S)n^{v(S)-2}p^{e(S)}$. It follows that

\begin{align*}
	m & > \alpha^2n_0^{v(S)}p^{e(S)}\\ 
	& >  \alpha^2 n_0^{s} C^{e(S)}n^{-s+2-1/m_2(H)}\\ 
	&> \alpha^2 C^{e(S)}\left(\dfrac{\mu}{sh}\right)^{s} n^{2-1/m_2(H)}.
\end{align*}
In the last inequality, we used that $n_0 > \mu |V_1| > \frac{\mu}{sh}n$.
Therefore, in order to apply Theorem \ref{thm:KLR}, we take 
\[C(\gamma):= \left(\left(\dfrac{sh}{\mu}\right)^s \dfrac{B(\gamma)}{\alpha^2}\right)^{1/e(S)}.\]
Observe that $C$ is, in fact, only a function of $\gamma$. Indeed, $\alpha \in (0,s^{-2f})$ is given by Lemma~\ref{prop:canspan} and it only depends on the original partition. That is, $\alpha$ is a function of $h$ and $s$.
The constant $\mu$ only depends on $\eps$ and $\alpha$, but $\eps$ depends on $\gamma$ and it is given by Theorem \ref{thm:KLR}.

Finally, by using the estimate $|\mathcal{G}^*|\le 2^{e(H)n}\gamma^m\binom{n_0^2}{m}^{e(H)}$, we obtain from~\eqref{almostfinal} that

\begin{align*}
	\P \big ( G \naoflecha F \oplus H \big ) & \le
	2^{e(H)n} \gamma^{m} \dbinom{n_0^2}{m}^{e(H)}q^{e(H)m} +o(1)\\
	& \le \gamma^{m}{\left(\frac{ 2e qn_0^2}{m}\right)}^{e(H)m} +o(1)\\
	&\leq \gamma^{m}{\left(\frac{2e}{\alpha^2}\right)}^{e(H)m}+o(1).
\end{align*}
In the second inequality above, we used that $n \ll m$ and that $\binom{a}{b}\le (\frac{ea}{b})^b$ for $a \ge b \ge 1$.  Choose $\gamma := (8e/\alpha^2)^{-e(H)}$.
Thus we conclude that 
\[ \P \big ( G \naoflecha F \oplus H \big ) \le 2^{-m}+o(1),\]
\noindent which finishes our proof.

\section{Book graphs}\label{sec:split}

Recall that, for a positive integer $t$, the book graph $B_t$ is defined by $t$ triangles sharing one edge. In this section we prove that book graphs fit the framework of Theorem~\ref{thm:main}. We start with the following result.

\begin{lemma}\label{lemma1}
	$B_{3t-2} \flecha B_{2,t}$ for every $t\geq 1$.
\end{lemma}

\begin{proof}
	The base case $t=1$ is trivial since every proper coloring of a triangle is rainbow. We assume the lemma to be true for every integer up to $t-1$ and we move one step in the induction.
	Let $\Phi$ be a proper-coloring of $B_{3t-2}$ and let $V(B_{3t-2}) = \{u_1, u_2, v_1, \ldots, v_{3t-2}\}$ where $\{u_1,u_2,v_t\}$ is a triangle for every $t\in [3i-2]$ and  $\{v_1, \ldots, v_{3t-2}\}$ is an independent set.
	By induction, we have that $\phi$ induces a rainbow copy of $B_{(t-1)}$ which, without loss of generality, we assume to be induced by $\{u_1, u_2, v_1, \ldots, v_{t-1}\}$.
	  
	Let $X$ be the set containing any $v_k$, with $t\leq k\leq 3t-2$, such that $\{u_1, u_2, \ldots, v_{t-1}, v_k\}$ does not induces a rainbow copy $B_{2,t}$. Since the coloring is proper, then $\Phi(u_iv_k)$ is different of $\Phi(u_1u_2)$ for every $i\in \{1,2\}$ and $v_k\in X$. Therefore, if $v_k$ belongs to $X$ then we must have that $\Phi(u_iv_k)=\Phi(u_{3-i}v_\ell)$ for some $i\in \{1,2\}$ and $\ell \in[t-1]$. For fixed $i\in \{1,2\}$ there can be at most $t-1$ values $k$ such that $\Phi(u_iv_k)=\Phi{u_{3-i}v_\ell}$ for some $\ell \in [t-1]$, since the coloring is proper. Therefore we have that $|X|<2t-2$ and we conclude that there exists a vertex that yields a rainbow copy of $B_t$.
\end{proof}

Now we are ready to prove Corollary~\ref{booksaregreat}.

\begin{proof}[Proof of Corollary~\ref{booksaregreat}]
Let $H$ be a graph with $m_2(H)\in (1,2)$. It is straightforward to check that for any $t\geq 1$ we have that $m_2(B_t)=2$, which implies that the hypothesis $1< m_2(H) <m_2(B_t)$ of Theorem~\ref{thm:main} is satisfied. Since $B_{3t-2}$ is a $2$-balanced graph and $B_{3t-2}\flecha B_t$, Theorem~\ref{thm:main} implies that that for some $B>0$ we have that 
\[\lim_{n \rightarrow \infty} \P\left[G(n,p) \flecha B_t \oplus H\right] = 1\]

\noindent whenever $p\geq Bn^{-\beta(H,B_{3t-2})}$. Now we only have to show that $\beta(H,B_{3t-2}) > 1/m_2(B_t\oplus H)$. Note that $v(B_{3t-2})=3t$ and $e(B_{3t-2})= 6t-3$. Since $m_2(H)<2$, then we have that
\[\beta(H,B_{3t-2})= \frac{1}{6t-3}\left(3t-2+\frac{1}{m_2(H)}\right)> \frac{1}{2}.\]

\noindent On the other hand, since $B_t$ is contained in $B_t\oplus H$, then $1/m_2(B_t\oplus H) \leq 1/2$, which finishes the proof.
\end{proof}

\section*{Acknowledgement}

This work began while the authors were visiting the Instituto Nacional de Matemática Pura e
Aplicada in Rio de Janeiro (IMPA) as part of the Graphs@IMPA thematic program. We are very grateful to IMPA and the organizers for the support to be able to attend the event and for the great working environment. We also thank Yoshiharu Kohayakawa for introducing us to the problem and for useful discussions in the beginning of this project.


\begin{thebibliography}{10}

\bibitem{BMS}
J.~Balogh, R.~Morris, and W.~Samotij.
\newblock Independent sets in hypergraphs.
\newblock {\em J. Amer. Math. Soc.}, 28(3):669--709, 2015.

\bibitem{BCMP}
G.~F. Barros, B.~P. Cavalar, G.~O. Mota, and O.~Parczyk.
\newblock Anti-ramsey threshold of cycles for sparse graphs.
\newblock {\em Electron. Notes Theor. Comput. Sci.}, 346:89--98, 2019.

\bibitem{ChGr}
F.~Chung and R.~Graham.
\newblock Quasi-random graphs with given degree sequences.
\newblock {\em Random Structures Algorithms}, 32(1):1--19, 2008.

\bibitem{Erdos79}
P.~Erd\H{o}s.
\newblock Some old and new problems in various branches of combinatorics.
\newblock {\em Proc. 10th southeast. Conf. Combinatorics, graph theory and
  computing}, 1(23):19--37, 1979.

\bibitem{GeSt}
S.~Gerke and A.~Steger.
\newblock The sparse regularity lemma and its applications.
\newblock In {\em Surveys in combinatorics 2005}, volume 327 of {\em London
  Math. Soc. Lecture Note Ser.}, pages 227--258. Cambridge Univ. Press,
  Cambridge, 2005.

\bibitem{janson}
S.~Janson.
\newblock Poisson approximation for large deviations.
\newblock {\em Random Structures \& Algorithms}, 1(2):221--229, 1990.

\bibitem{KKM14}
Y.~Kohayakawa, P.~B. Konstadinidis, and G.~O. Mota.
\newblock On an anti-{R}amsey threshold for random graphs.
\newblock {\em European J. Combin.}, 40:26--41, 2014.

\bibitem{KKM18}
Y.~Kohayakawa, P.~B. Konstadinidis, and G.~O. Mota.
\newblock On an anti-{R}amsey threshold for sparse graphs with one triangle.
\newblock {\em J. Graph Theory}, 87(2):176--187, 2018.

\bibitem{KK}
Y.~Kohayakawa and B.~Kreuter.
\newblock Threshold functions for asymmetric {R}amsey properties involving
  cycles.
\newblock {\em Random Structures Algorithms}, 11(3):245--276, 1997.

\bibitem{KMPS}
Y.~Kohayakawa, G.~O. Mota, O.~Parczyk, and J.~Schnitzer.
\newblock The anti-ramsey threshold of complete graphs.
\newblock 2019.

\bibitem{Di}
C.~McDiarmid.
\newblock On the method of bounded differences.
\newblock In {\em Surveys in combinatorics, 1989 ({N}orwich, 1989)}, volume 141
  of {\em London Math. Soc. Lecture Note Ser.}, pages 148--188. Cambridge Univ.
  Press, Cambridge, 1989.

\bibitem{NPSS}
R.~Nenadov, Y.~Person, N.~\v{S}kori\'{c}, and A.~Steger.
\newblock An algorithmic framework for obtaining lower bounds for random
  {R}amsey problems.
\newblock {\em J. Combin. Theory Ser. B}, 124:1--38, 2017.

\bibitem{RR95}
V.~R\"{o}dl and A.~Ruci\'{n}ski.
\newblock Threshold functions for ramsey properties.
\newblock {\em J. Amer. Math. Soc.}, 8(4):917--942, 1995.

\bibitem{RT}
V.~R\"{o}dl and Z.~Tuza.
\newblock Rainbow subgraphs in properly edge-colored graphs.
\newblock {\em Random Structures Algorithms}, 3(2):175--182, 1992.

\bibitem{Spencer}
J.~Spencer.
\newblock Counting extensions.
\newblock {\em Journal of Combinatorial Theory, Series A}, 55(2):247--255,
  1990.

\end{thebibliography}
\end{document}